\documentclass[12pt]{amsart}
\usepackage{pgfplots}
\usepackage{float}
\usepackage[utf8]{inputenc}
\usepackage{graphicx}
\usepackage{color}
\usepackage{mdwlist}
\usepackage{soul,xcolor}
\usepackage{ulem}
\usepackage{graphics}
\allowdisplaybreaks

\usepackage{cleveref}

\newcommand{\beq}{\begin{eqnarray*}}
\newcommand{\feq}{\end{eqnarray*}}
\newcommand{\beqn}{\begin{eqnarray}}
\newcommand{\feqn}{\end{eqnarray}}

\newtheorem{theorem}{Theorem}[section]
\newtheorem{lemma}[theorem]{Lemma}
\newtheorem{corollary}[theorem]{Corollary}
\newtheorem{proposition}[theorem]{Proposition}
\theoremstyle{definition}

\theoremstyle{remark}
\newtheorem{remark}[theorem]{Remark}
\numberwithin{equation}{section}

\setstcolor{red}
\newtheorem*{theorem*}{Theorem}

\begin{document}
\title[Critical thresholds in Euler-Poisson systems]{Critical thresholds in 
1D pressureless Euler-Poisson systems with varying background}

\author{Manas Bhatnagar and Hailiang Liu}
\address{Department of Mathematics, Iowa State University, Ames, Iowa 50011}
\email{manasb@iastate.edu}
\email{hliu@iastate.edu} 
\keywords{Euler-Poisson system, critical threshold, global regularity, shock formation}
\subjclass{Primary, 35L65; Secondary, 35L67} 

\begin{abstract} 
The Euler Poisson equations describe important physical phenomena in many applications such as semiconductor modeling and plasma physics. This paper is to advance our understanding of critical threshold phenomena in such systems in the presence of different forces.  We identify critical thresholds in 
two damped Euler Poisson systems, with and without alignment, both with attractive potential and spatially varying background state. For both systems, we give respective bounds for subcritical and supercritical regions in the space of initial configuration, thereby proving the existence of a critical threshold for each scenario. Key tools include comparison with auxiliary systems, phase space analysis of the transformed system. 
\end{abstract}

\maketitle

\section{Introduction} \label{prob} 
We are concerned with the critical threshold phenomenon in Euler-Poisson equations subject to 
local and nonlocal forces.
\subsection{Euler Poisson equations}
Euler Poisson equations have been an area of intensive study due to their vast relevance in modeling physical phenomena \cite{BRR94, HJL81, Ja75, Ma86, MaPe90, MRS90}. The general system is composed of three sets of equations: the mass conservation equation, the momentum equations and the Poisson equation. The system 
\begin{align*}
& \rho_t + \nabla\cdot (\rho\mathbf{u}) = 0,\\
& \mathbf{u}_t + \mathbf{u}\nabla\mathbf{u} + \frac{\nabla P(\rho)}{\rho} =-\nu \mathbf{u}   -k\nabla\phi,\\
& -\Delta \phi = \rho - c
\end{align*}
governs the unknown density $\rho=\rho(t, x)$  and velocity $\mathbf{u}=\mathbf{u}(t, x)$ for $x\in \mathbb{R}^N$ (or a bounded domain) and time $t>0$, subject to initial conditions $\rho(0, x)$ and 
$\mathbf{u}(0, x)$. $P$ is the pressure and $c=c(x)$ is the background state which varies with the space variable. The parameter $k$ signifies the property of the underlying force, repulsive $k>0$ or attractive $k<0$, governed by the Poisson equation through the potential $\phi$.

Such systems are widely used in semiconductor modeling where the charge density and current need to be modeled. $\phi$ then represents the electric potential and hence, $-\nabla\phi$ is the electric field. $c$ represents the background charge the semiconductor is doped with. This could be a constant or vary with position within the semiconductor; see \cite{MRS90}.  Another widely known application of this system is modeling plasmas dynamics \cite{Ja75}. Here, the pressure forcing plays an important role, which is usually adiabatic  of form $P(\rho) = A\rho^\gamma,\ \gamma\geq 1$. 

An addition of convolution terms on the right hand side of the momentum equation gives rise to a different class of systems with nonlocal forcing. Such systems have primarily been studied without pressure. It is then called the Euler alignment/Euler Poisson alignment system for $k=0$ and $k\neq 0$ respectively. The momentum equation reads
$$
\mathbf{u}_t + \mathbf{u}\nabla\mathbf{u} = -k\nabla\phi + \psi\ast(\rho\mathbf{u})-\mathbf{u}\psi\ast\rho.
$$
Euler alignment systems arise as macroscopic realization of agent-based dynamics \cite{CS07a,CS07b} which describes the collective motion of finite agents, each of which adjusts its velocity to a weighted average of velocities of its neighbors
\begin{align*}
& \dot x_i =v_i, \\
& \dot v_i= \frac{1}{N} \sum_{j=1}^N \psi(|x_i-x_j|)(v_j-v_i).
\end{align*}
Here $\psi$ is often called influence potential. 
See \cite{HaTa08} for realization of Euler alignment system as a mean field limit of the above type finite agent model as $N\to\infty$. 
It is known that global-in-time strong solutions for the hydrodynamic alignment system will flock.  Global regularity or critical thresholds for such systems 
have been analyzed extensively during the recent years, see \cite{HeTa19, TT14}. Further relevant literature is discussed in Section \ref{pastwork}. 

\subsection{Critical threshold phenomena}
It is well known that the finite-time breakdown of the systems of Euler equations for compressible flows is generic in the sense that finite-time shock formation occurs for all but a “small” set of initial data. Lax \cite{La64} showed that for pairs of conservation laws, $C^1$-smoothness of solutions can be lost unless its two Riemann invariants are non decreasing. With the additional Poisson forcing the system of Euler–Poisson equations admits a “large” set of initial configurations which yield global smooth solutions, see, e.g. 
\cite{ELT01, Lee2, LT02a, LT03, TW08}. 
Indeed, for a class of pressureless Euler--Poisson equations, the question addressed in \cite{ELT01} is whether there is a critical threshold for the initial data such that the persistence of the $C^1$ solution regularity depends only on crossing such a critical threshold. For example, for system of Euler--Poisson equations with only electric force,
\begin{align*}
\begin{aligned}
&\rho_t + (\rho u)_x = 0, \\
&u_t + uu_x =  -k\phi_x , \\
&-\phi_{xx} = \rho -c.
\end{aligned}
\end{align*}
It was shown in \cite[Theorem 3.2]{ELT01} that the system with $k<0$ admits a global solution iff 
$$
u_{0x}(x) \geq \sqrt{-\frac{k}{c}}(\rho_0(x)-c) \; \forall x \in\mathbb{R}, 
$$
and for $k>0$, the critical threshold condition becomes 
$$
|u_{0x}(x)| <\sqrt{k(2\rho_0(x)-c)} \quad c=\text{const}>0. 
$$
It is evident that in the attractive forcing case ($k<0$), the background has a balancing effect and hence, 
it may be possible to observe some comparison principle. In the repulsive forcing case ($k>0$), the solution of the system will be oscillatory, hence more subtle to analyze.  When the above mentioned cases are augmented with damping forces in the momentum equation, an enlarged subcritical region may be observed 
due the further balancing effect from damping. A novel phase plane analysis method was introduced in \cite{BL20} to identify sharp critical thresholds in various scenarios.   


Even amidst the vast study of critical thresholds in Euler-Poisson systems, the variable background case has not been studied much. In fact, to our best knowledge, there is no known critical threshold result for Euler-Poisson equations with a background state that varies in space. This paper is devoted to the study of such scenario.
\subsection{Present investigation} In this work we focus on the pressureless case, in one dimensional periodic setting. Without loss of generality, we can set $\mathbb{T} := [-1/2,1/2]$ to be the domain of the spatial variable. More precisely, we consider the following  damped  Euler--Poisson system with potential induced by a background which is a function of the space variable, 
\begin{subequations}
\begin{align}
\label{EPMain}
\begin{aligned}
&\rho_t + (\rho u)_x = 0, \\
&u_t + uu_x = -\nu u -k\phi_x , \\
&-\phi_{xx} = \rho -c(x),\\
\end{aligned}
\end{align}
on $(0,\infty )\times\mathbb{T}$ subject to periodic initial conditions,
\begin{align}
\label{EPMain2}
\begin{aligned}
& \rho (0,x) = \rho_0 (x) \geq 0,\qquad \rho_0 \in C^1 (\mathbb{T}),\\
& u(0,x) = u_0 (x),\qquad u_0 \in C^1 (\mathbb{T}),\\
\end{aligned}
\end{align}
\end{subequations}
where $c(x)$ is the periodic background term which is Lipschitz continuous and satisfies $0< c_1\leq c\leq c_2$, $\nu \geq 0$ is the damping coefficient,  and parameter $k<0$ signifies attractive forcing. 

Furthermore, we add a nonlocal forcing to the momentum equation. The resulting system is called an Euler-Poisson alignment system,
\begin{subequations}
\begin{align}
\label{backalign}
\begin{aligned}
&\rho_t + (\rho u)_x = 0, \\
&u_t + uu_x = -\nu u + \psi\ast(\rho u)-u\psi\ast\rho -k\phi_x, \\
&-\phi_{xx} = \rho -c(x),\\
\end{aligned}
\end{align}
on $(0,\infty )\times\mathbb{T}$, subject to periodic initial conditions,
\begin{align}
\label{backalign2}
\begin{aligned}
& \rho (0,x) = \rho_0 (x) \geq 0,\qquad \rho_0 \in C^1 (\mathbb{T}),\\
& u(0,x) = u_0 (x),\qquad u_0 \in C^1 (\mathbb{T}),\\
\end{aligned}
\end{align}
\end{subequations}
where $\psi:\mathbb{R}\to [0,\infty )$ is assumed to have the following properties,
\begin{itemize}
\item
$\psi(x) = \psi(-x),\ \forall x>0\ $(Symmetric), 
\item
$\psi(x+1)=\psi(x),\ \forall x\in\mathbb{R}\ $ (1-periodic),
\item
$|\psi(x)-\psi(y)|\leq K|x-y|,\ x,y\in\mathbb{R}$ and some $K>0\ $ (Lipschitz continuous). 
\end{itemize}
Let $\min \psi = \psi_m$ and $\max \psi = \psi_M$.
Also, without loss of generality, 
$$
 \int_\mathbb{T}\rho_0(x)\, dx= \int_\mathbb{T}\rho(t,x)\, dx = 1 \; \text{and}\;  \int_\mathbb{T} c(x)dx = 1.
 $$ 
In this paper, we obtain bounds for supercritical as well as subcritical regions in the configuration of initial data for the aforementioned cases, thereby proving the existence of a critical threshold for each system.

\subsection{Related work}\label{pastwork}
There is a considerable amount of literature available on the solution behavior of Euler--Poisson equations. 
\cite{En96, WC98} gives results for nonexistence and singularity formation;   \cite{CW96, WW06} for global existence of weak solutions with geometrical symmetry; \cite{MN95} for isentropic case, and \cite{PRV95} for isothermal case. For 3-D irrotational flow consult \cite{Gu98, GMP13, GP11}. Smooth irrotational solutions for the two dimensional Euler--Poisson system are constructed independently in \cite{IP13, LW14}. See also \cite{Ja12, JLZ14} for related results on two dimensional case. In the one-dimensional Euler--Poisson system with both adiabatic pressure and a nonzero background, the authors in \cite{GHZ17} showed the persistence of global solutions for initial data which is a small perturbation about the equilibrium. 
Yet the existence of a critical threshold for such setting is still open. 

For results on critical thresholds in restricted Euler-Poisson systems, we refer to \cite{LT03} for sharp conditions on global regularity vs finite time breakdown for the 2-D restricted Euler--Poisson system, and \cite{LT02a} for sufficient conditions on finite time breakdown for the general n-dimensional restricted Euler--Poisson systems. A relative complete analysis of critical thresholds in multi-dimensional restricted Euler--Poisson systems is given in \cite{Lee2} for both attractive and repulsive forcing. For multidimensional Euler-Poisson with spherically symmetric solutions, see \cite{ELT01,WTB12}. 

During recent years, Euler alignment systems have been studied by several researchers,
 see \cite{CCTT16, TT14} for alignment forces dictated by bounded kernels, \cite{DKRT18, TK18} by singular kernels.  The authors in \cite{TT14} give bounds on subcritical and supercritical regions for the Euler alignment system, i.e., $k=0, \nu=0$ in (\ref{backalign}) with bounded kernel in one and two dimensions. The critical threshold condition for one dimensional Euler alignment system was further made precise in \cite{CCTT16}. The authors also studied undamped Euler Poisson alignment system, i.e. $k\neq 0$, $\nu = 0$, $c=0$ in \eqref{backalign} where they showed that for such a system with $k<0$, there is unconditional breakdown. Our result stated in Theorem \ref{nonlocvarcgs} shows the existence of a subcritical region in case of a positive background. Therefore, once again we see that presence of background with attractive forcing has a balancing effect.  However, our comparison tools do not seem to be applicable to the situation with repulsive forcing ($k>0$). 
 
\subsection{Plan of the paper}
Our work analyses two classes of Euler Poisson systems: 
\begin{enumerate}
	\item Pressureless Euler Poisson with background, and 
	\item Pressureless Euler Poisson alignment with background.
\end{enumerate}
As a result all the further components of this paper are divided into two parts, each pertaining to a system. Section \ref{mainresults} contains the main results along with some necessary preliminary analysis. It has three subsections. The first one is devoted to the preliminary calculations. The other two contain the main results for each of the aforementioned systems. Section \ref{epwoalign} contains the analysis/tools/proof to the theorems pertaining to the first system and Section \ref{varbackalign} contains the same for Euler Poisson alignment system.

\section{Main results}
\label{mainresults}
\subsection{Preliminaries}
\label{prelim1} 
The critical threshold analysis to be carried out is the a priori estimate on smooth solutions as long as they exist.  For the one-dimensional
Euler-Poisson problem, local existence of smooth solutions was long known, it can be justified by using the characteristic method in the pressureless case.

\begin{theorem}
\label{local}
$($\textbf{Local existence}$)$ If $\rho_0 \in C^1$ and  $u_0 \in C^1$ , then there exists $T>0$, depending on the initial data, such that the initial value problem (\ref{EPMain}),  (\ref{EPMain2}) admits a unique solution $(\rho, u) \in C^{1} ([0,T)\times \mathbb{T}).$  Moreover,  if the maximum life span $T^* < \infty$, then
$$
\lim_{t \uparrow T^*} \partial_x u(t, x^*) =-\infty
$$
for some $x^* \in \mathbb{T}$.
\end{theorem}
To our knowledge, such local existence theorem has been known for a constant background case ($c=$const  in \eqref{EPMain}). However, we will formally justify that the dependence of $c$ on the space variable does not change the result of the theorem for \eqref{EPMain} as well as \eqref{backalign} as long as $c(x(t))$ is well-defined and bounded.  We will show this by analyzing a set of equations obtained along the characteristic curve. 

We proceed to derive the characteristic system which is essential to our critical threshold analysis. Differentiate the second equation in (\ref{EPMain}) with respect to $x$,  and set $d:=u_x$ to obtain:
	\begin{subequations}\label{nonlin}
	\begin{align}
	& \rho' + \rho d = 0, \label{nonlin11} \\
	& d' + d^2 + \nu d = k(\rho -c(x(t)) ), \label{nonlin12}
	\end{align}
	\end{subequations}
	where we have used the Poisson equation in (\ref{EPMain})  for $\phi$, and 
	\[
	\{\}' = \frac{\partial}{\partial t} + u\frac{\partial}{\partial x}
	\]
	 denotes the differentiation along the particle path,
	$$
	\Gamma=\{(t,x)|\;  x'(t)=u(t,x(t)), x(0)=\alpha \in \mathbb{T}\}.
	$$
Here, we employ the method of characteristics to convert the PDE system (\ref{EPMain}) to ODE system (\ref{nonlin}) along the particle path which is fixed for a fixed value of the parameter $\alpha$. Consequently, the initial conditions to the above equations are $\rho (0)=\rho_0 (\alpha)$ and $d(0) = d_0 (\alpha) = u_{0x}(\alpha )$ for each $\alpha\in\mathbb{T}$. Note that this in itself is not a closed system.  However, we can obtain a complete system with additional ODEs.  Setting 
$$
E:=-\phi_x = \int_{-1/2}^{x}\!\!\!\rho(t,y) - c(y)dy-\int_0^t (\rho u)(s, -1/2)ds, 
$$ 
we obtain,
\begin{align}
\label{irrsys}
\begin{aligned}
& x' = u,\\
& u' + \nu u = kE,\\
& E' = -cu. 
\end{aligned}
\end{align}
For $E$ to be periodic, necessarily 
$$
\int_\mathbb{T}\rho(t, y) - c(y)dy =0 
$$
for any $t>0$. This combined with the conservation of mass requires 
 $\int_\mathbb{T}\rho_0(y) - c(y)dy = 0$. Note that this subsystem is a closed system, which allows us to independently analyze \eqref{nonlin} with $c$ obtained from system (\ref{irrsys}).  Since $c(x)$ is Lipschitz continuous, the system for $(x, u, E)$ admits a unique solution for each given initial data.   Moreover,
\begin{align*}
\frac{1}{2}(x^2+u^2+E^2)' & = x u+kuE-\nu u^2- cuE\\
& \leq (1+|k|+2\nu+\max c)(x^2+u^2+E^2).
\end{align*}
On integrating we get 
$$
x^2+u^2+E^2 \leq (\alpha^2+u_0^2+E^2(0, \alpha))e^{2(1+|k|+2\nu+\max c)t} \quad \forall t>0,
$$
which says that $x,u,E$ remain bounded for all time. Hence, we can solely analyze \eqref{nonlin} to conclude the long time behavior of the solution. The system with alignment \eqref{backalign} is no different. Indeed, $u$ remains bounded because at any time, the alignment force is a mere weighted average of the relative speed. 

The above discussion shows that we still lie in the purview of Theorem \ref{local}. That is, if $u_x$ remains bounded for all the characteristics then we ensure global-in-time solution from Theorem \ref{local}. Likewise, if for any characteristic, $|u_x|\to\infty$ in finite time, there is finite time breakdown. This allows us to analyze \eqref{nonlin} as a system for our purpose using the bounds of $c(x)$.

 
We are now in a position to establish our critical threshold theory which includes results on both the global-in-time solution and finite time breakdown for \eqref{EPMain} and \eqref{backalign}. 

In order to conveniently present our main results and their proofs, we here introduce two functions 
from $\mathbb{R}^+\times \mathbb{R}^+ \to \mathbb{R}$,
\begin{subequations} \label{roots}
\begin{align}
& \Omega(\gamma, \beta):= \frac{\beta+\sqrt{\beta^2-4k\gamma}}{2}, \\ 
& \Theta(\gamma, \beta):= \frac{-\beta+\sqrt{\beta^2-4k\gamma}}{2}, 
\end{align}  
\end{subequations} 
where these two functions are both non-negative, and  
$$
-\Omega(\gamma, \beta) \leq 0 \leq \Theta(\gamma, \beta).
$$
 

\subsection{Euler-Poisson with variable background}
\label{secdampbacknoalign}
We first state our main results for \eqref{EPMain}. 
\begin{theorem}[Global Solution]
\label{GS}
Consider the system \eqref{EPMain} with initial conditions \eqref{EPMain2}. Let $k<0$, $\lambda_1 = \Omega(c_1, \nu)$, $0<c_1=\min_{x\in \mathbb{T}} c(x)$ and $c_2=\max_{x\in \mathbb{T}} c(x)$.  If  
$$
(\rho_0(x),u_{0x}(x))\in\left\{ (\rho ,d):  d> \frac{\lambda_1}{c_1}(\rho -c_1) \right\} \quad \forall x, 
$$ 
then there exists a unique global-in-time solution, $\rho,u\in C^1((0,\infty)\times\mathbb{T})$. In particular, 
\begin{align*}
& ||\rho(t,\cdot)||_\infty\leq ||\rho_0||_\infty e^{\lambda_1 t},\ and\\ 
& -\lambda_1\leq u_x(t,\cdot)\leq\max\left\{ \max u_{0x}, \Theta(c_2, \nu)\right\}. 
\end{align*}
\end{theorem} 

\begin{theorem}[Finite Time Breakdown]
\label{FTB}
Consider the system \eqref{EPMain} with initial conditions \eqref{EPMain2}. Let $k<0$ and $c_2=\max_{x\in \mathbb{T}} c(x)$. If $\exists x_0\in\mathbb{T}$ such that
$$
u_{0x}(x_0) < \frac{\Omega(c_2, \nu)}{c_2}(\rho_0(x_0)-c_2), 
$$
then $\exists (t^*,x^*)$ such that $\lim_{t\uparrow t^{*}} u_x(t,x^*)=-\infty$.
\end{theorem}
\begin{remark}
We would like to point out that essentially the same threshold results hold for the case when the domain is all of $\mathbb{R}$. We work on the periodic case to avoid a technical discussion at far fields ($x\to\pm\infty$), 
especially in the alignment with background case (Section \ref{varbackalign}), where talking about far fields is physically less meaningful as the total mass is infinite because of the following imperative neutrality condition
\[
\int_{-\infty}^\infty (\rho_0(y)-c(y))\, dy = 0.
\]
\end{remark}
\begin{theorem}
\label{cconst}
Consider the system \eqref{EPMain} with initial conditions \eqref{EPMain2} and assume $c(x)\equiv c$. Then there exists unique global solution $\rho,u\in C^1((0,\infty)\times\mathbb{T})$ iff 
$$
(\rho_0(x),u_{0x}(x))\in\left\{ (\rho ,d): d\geq \frac{\Omega(c,\nu)}{c}(\rho -c) \right\} \quad \forall x \in \mathbb{T}.
$$
\end{theorem}

\subsection{Euler-Poisson alignment with variable background}
We now state the main results for \eqref{backalign}. 
\begin{theorem}
\label{nonlocvarcgs}
Consider the system \eqref{backalign} with initial conditions \eqref{backalign2}. Let $k<0$, $0<c_1 = \min_{x\in\mathbb{T}} c(x)$, $c_2 = \max _{x\in\mathbb{T}} c(x)$, $\psi_M = \max_{x\in\mathbb{R}}\psi(x)$, $\psi_m = \min_{x\in\mathbb{R}}\psi(x)$ and \\
$\lambda_M = \Omega(c_1,\nu+\psi_M)$.
If 
$$
u_{0x}(x)> \frac{\lambda_M\rho_0(x)}{c_1} - \Theta(c_1,\nu+\psi_M) -\nu- \psi\ast\rho_0(x) \quad 
\forall x\in\mathbb{T},
$$	
then there exists a unique global solution $\rho,u\in C^1((0,\infty)\times\mathbb{T})$.
Furthermore, we have the following bounds,
\begin{align*}
& ||\rho(t,\cdot)||_\infty\leq ||\rho_0||_\infty e^{\lambda_M t},\\
& -\lambda_M\leq u_x(t,\cdot )\leq \max\{\max u_{0x}, \Theta(c_2,\nu+\psi_M) \}+\psi_M-\psi_m.
\end{align*}
\end{theorem}

\begin{theorem}
\label{nonlocvarcftb}
Consider the system \eqref{backalign} with initial conditions \eqref{backalign2}. Let $k<0$, $c_2 = \max _{x\in\mathbb{T}} c(x)$ and $\psi_m = \min_{x\in\mathbb{R}}\psi(x)$. If $\exists x_0\in\mathbb{T}$ such that,
$$
u_{0x}(x_0)<\frac{\Omega(c_2,\nu+\psi_m)\rho_0(x_0)}{c_2}-\Theta(c_2,\nu+\psi_m)-\nu-\psi\ast\rho_0(x_0),
$$	
then $\inf\lim_{t\uparrow t_c}u_x(t,\cdot ) = -\infty$ for some finite $t_c$. 
\end{theorem}
\begin{remark}
We would like to point out that if $\psi\equiv 0$ in the above theorems then we recover Theorem \ref{GS} and \ref{FTB},  respectively. In other words our analysis of the generalization to the case with alignment is optimal.   
Also, if $\psi(x)\equiv \psi$ is a constant, then we obtain the following two corollaries.
\end{remark}
\begin{corollary}
Consider the system \eqref{backalign}, initial conditions \eqref{backalign2} with $\psi(x)\equiv\psi$ (constant). Let $k<0$, $0<c_1 = \min_{x\in\mathbb{T}} c(x)$, $c_2 = \max _{x\in\mathbb{T}} c(x)$ and $\lambda_1 = \Omega(c_1,\nu+\psi)$. If $$
u_{0x}(x)> \frac{\lambda_1}{c_1}(\rho_0(x) - c_1) \; \forall x\in\mathbb{T},
$$	
then there exists a unique global-in-time solution $\rho,u\in C^1((0,\infty)\times\mathbb{T})$. 
Furthermore, we have the following bounds,
\begin{align*}
& ||\rho(t,\cdot)||_\infty\leq ||\rho_0||_\infty e^{\lambda_1 t},\\
& -\lambda_1\leq u_x(t,\cdot )\leq \max\{\max u_{0x}, \Theta(c_2,\nu+ \psi) \},
\end{align*}
\end{corollary}

\begin{corollary}
Consider the system \eqref{backalign}, initial conditions \eqref{backalign2} with $\psi(x)\equiv\psi$ (constant). Let $k<0$, $c_2 = \max _{x\in\mathbb{T}} c(x)$. If $\exists x_0\in\mathbb{T}$ such that,
$$
u_{0x}(x_0)<\frac{\Omega(c_2,\nu+\psi)}{c_2}(\rho_0(x_0)-c_2),
$$	
then $\inf\lim_{t\uparrow t_c}u_x(t,\cdot ) = -\infty$ for some finite $t_c$. 
\end{corollary}

\section{Euler-Poisson systems with variable background}
\label{epwoalign}
\subsection{Critical thresholds for an auxiliary system} 
The main tool in dealing with the variable background is the use of comparison. To this end,  we introduce an auxiliary ODE system corresponding to \eqref{nonlin},
\begin{subequations}
\label{auxnonlin}
\begin{align}
& \eta'=-\eta\xi,\label{auxnonlin1}\\
& \xi' = -\xi^2-\nu\xi + k\eta - k\gamma. \label{auxnonlin2}
\end{align}
\end{subequations}
where $\gamma\geq 0$ is a parameter. Hence, $\eta ,\xi$ are functions of time as well as the parameter $\gamma $. However, we will omit the latter dependence on the parameter whenever it is clear from context. We will make use of the phase plane analysis technique introduced in \cite{BL20} to prove a proposition for this auxiliary problem which will play a crucial role in proving the theorems stated.
\begin{proposition}
\label{phpl}
Consider the ODE system \eqref{auxnonlin} with initial conditions $(\eta(0)\geq 0,\xi(0))$, then $0\leq\eta(t)$ and $\xi(t)\leq\max\{ \xi(0),\Theta(\gamma,\nu) \}$ for all $t>0$.  The solution exists globally for all $t>0$ with 
$$\eta(t) \leq \eta(0)e^{\lambda t},\ 
and\ -\lambda \leq \xi(t),
$$ 
if and only if 
$$
\xi(0) \geq \frac{\lambda}{\gamma}(\eta(0)-\gamma).
$$
Here
$\lambda = \Omega(\gamma,\nu)$. Moreover, if 
$$
\xi(0) < \frac{\lambda}{\gamma}(\eta(0)-\gamma),
$$
then $\lim_{t\to t_c^-}\eta(t)= -\lim_{t\to t_c^-}\xi(t) =\infty $ for some $t_c>0$.
\end{proposition}
\begin{proof} Note that from \eqref{auxnonlin1}, we have $\eta(t) = \eta(0)e^{-\int_0^t \xi\, d\tau}$ 
and hence, if $\eta(0)>0$ then $\eta(t)>0$ for all $t>0$ as long as the solution exists 
and $\eta(0)=0\implies \eta(t)\equiv 0$. Hence, $\eta$ maintains sign.  For a uniform upper bound on $\xi$, note that since $\eta\geq 0$, from \eqref{auxnonlin2},
\begin{align*}
\xi' & \leq -\xi^2 - \nu\xi - k\gamma\\
& = -(\xi + \lambda)(\xi-\mu).
\end{align*}
 with $\lambda = \Omega(\gamma,\nu)$ and $\mu = \Theta(\gamma,\nu)$ satisfying $-\lambda\leq 0 \leq\mu$. Comparing the above inequality with \eqref{extra1}, we obtain $\xi(t)\leq \max\{ \xi(0), \mu \}$ for all $t>0$. 

We first consider the case when $\eta(0)=0\equiv\eta(t)$. 
Then from \eqref{auxnonlin2}, 
\begin{align}
\label{extra1}
\xi' = -\xi^2-\nu\xi - k\gamma =: -(\xi + \lambda)(\xi-\mu).
\end{align}
Using phase line analysis, we have that $\xi(t)$ exists for all time if and only if 
$\xi(0) \geq -\lambda$. In the case of global solution, we have 
$$
 -\lambda \leq \xi(t) \leq \max\{\xi(0), \mu\}.  
$$
If $\xi(0)< -\lambda$, $\xi(t)$ tends to $-\infty$ at a finite time $t_c$. Such time may be determined from 
the solution formula of form 
$$
\frac{\xi-\mu}{\xi+\lambda}=\frac{\xi(0)-\mu}{\xi(0)+\lambda}e^{(\lambda+\mu)t},
$$
we have 
$$
t_c=\frac{1}{\mu+\lambda} \log \left| \frac{\xi(0)+\lambda}{\xi(0)-\mu}\right|.
$$
As a result of the above discussion, we can now assume $\eta(0)>0$ which in turn implies $\eta(t)>0$ for all $t>0$. We proceed to introduce the transformation 
$$
r:=\xi/\eta, \quad s:=1/\eta,
$$
 so that \eqref{auxnonlin} is transformed to the following linear system,
\begin{subequations}
\label{lin}
\begin{align}
& r'=-\nu r + k - k\gamma s,\label{lin11}\\
& s' = r  \label{lin12}
\end{align}
\end{subequations}
with initial data $r(0):=\xi(0)/ \eta(0)$ and $s(0) := 1/\eta(0)>0$. This is a linear ODE system, its solution 
$(r,s)$ will remain bounded for all time. This fact when combined with the transformation says that $(\xi, \eta)$
exists globally if and only if $s(t)>0$ for all time.  
Therefore, the key here is to identify critical thresholds for initial data to ensure $s(t)>0$ for all positive times. 

We move onto analysing \eqref{lin}. It is a linear system with the critical point $(0,1/\gamma)$ being the saddle point. Written in matrix form, the system is:
$$
\begin{bmatrix}
r\\
s-1/\gamma	
\end{bmatrix}' = 
\begin{bmatrix}
-\nu & -k\gamma \\
1 & 0	
\end{bmatrix}
\begin{bmatrix}
r\\
s-1/\gamma	
\end{bmatrix}
$$
The coefficient matrix has eigenvalues $-\lambda$ and $\mu$. 
Hence the general solution to this system is,
\begin{align}
\label{linsystemexp}
\begin{aligned}
& \begin{bmatrix}
r\\
s-\frac{1}{\gamma}	
\end{bmatrix} = A
\begin{bmatrix}
-\lambda\\
1	
\end{bmatrix} e^{-\lambda t}
+ B
\begin{bmatrix}
\mu\\
1	
\end{bmatrix} e^{\mu t}.	
\end{aligned}
\end{align}
From the flow of solution trajectories we see that the separatrix with incoming trajectories serves to divide the upper half plane($s>0$) into two invariant regions, one of which has the property that if $s(0)>0$,  then $s(t)>0$ for all $t>0$. 

Such separatrix corresponds to the special solutions with $B=0$, i.e.,
\begin{align*}
& \begin{bmatrix}
r\\
s-\frac{1}{\gamma}	
\end{bmatrix} = A
\begin{bmatrix}
-\lambda\\
1	
\end{bmatrix} e^{-\lambda t}.	
\end{align*}
Consequently, this trajectory equation is, 
$$
\gamma r = \lambda (1-\gamma s). 
$$ 
\begin{figure}[H]
 \centering
 \includegraphics[width=0.6\linewidth]{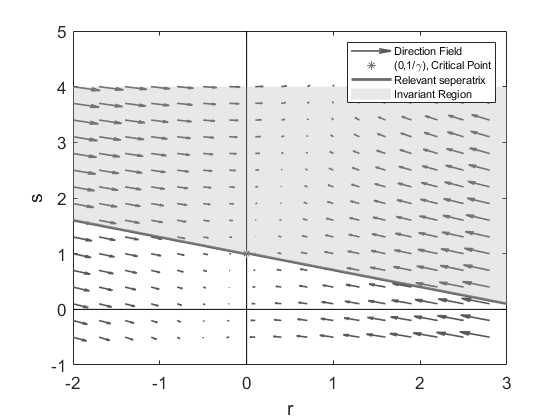}
 \caption{Direction field for reduced linear system along with the invariant region.($\nu =3,k=-1,\gamma =1$)}
 \label{fig1}
\end{figure}
Thus the above mentioned region can be characterized by 
$$
\Sigma_\gamma := \left\{(r,s): \gamma r\geq \lambda (1-\gamma s),\ s> 0 \right\}.
$$
In order to see this is an invariant region, we only need to show that on $\partial \Sigma_\gamma \cap\{s=0\}$ the trajectories go into the region. Note that the $r-$intercept of the separatrix is $(\lambda /\gamma) >0$, 
hence for $r\geq \lambda /\gamma$ and $s=0$ we have $s'=r> 0$ and the trajectory travels upwards.
 
Moving back to the original variables $(\eta, \xi)$, 
$\Sigma_\gamma$ transforms to 
$$
\widetilde{\Sigma}_\gamma :=\left\{ (\eta,\xi):\xi\geq \frac{\lambda}{\gamma}(\eta-\gamma),\ \eta>0 \right\}.
$$
Likewise, if $(\eta(0),\xi(0))\in\widetilde{\Sigma}_\gamma$, then $\eta(t)<\infty$ and $\xi(t)\geq\frac{\lambda}{\gamma}(\eta(t)-\gamma)\geq-\lambda$ for all $t>0$.

Now suppose $(r(0),s(0)>0)\notin\Sigma_\gamma$. 
Since, the linear ODE system has only one critical point, we have that $\lim_{t\to\infty}(|r(t)|,s(t))= (\infty,-\infty)$. Hence, the solution crosses $s=0$ line at some finite time, $t_c$.  

We will now derive an upper bound on $t_c$. Using the general solution (\ref{linsystemexp}) we can find the solution formula,  
$$
s(t) = \frac{1}{\gamma} + \frac{\mu}{(\lambda+\mu)}\left( s(0)-\frac{1}{\gamma}-\frac{r(0)}{\mu} \right)e^{-\lambda t} + \frac{\lambda}{(\lambda+\mu)}\left(s(0)-\frac{1}{\gamma}+\frac{r(0)}{\lambda} \right)e^{\mu t}.
$$
Assuming the finite time breakdown condition, i.e., $s(0)-\frac{1}{\gamma}+\frac{r(0)}{\lambda} <0$, then 
$$
s(t)\leq \frac{1}{\gamma} + \left( s(0)+\frac{1}{\gamma}+\frac{|r(0)|}{\mu} \right) - \frac{\lambda}{(\lambda+\mu)}\left|s(0)-\frac{1}{\gamma}+\frac{r(0)}{\lambda} \right|e^{\mu t}.
$$
Consequently, $s(t_c) = 0$ for some 
$$
t_c\leq \frac{1}{\mu}\ln\left( \frac{(\lambda+\mu)}{\lambda}\left( \frac{s(0)+\frac{2}{\gamma}+ \frac{|r(0)|}{\mu}}{\left| s(0)-\frac{1}{\gamma}+\frac{r(0)}{\lambda} \right|} \right) \right).
$$
And therefore, if $(\eta(0),\xi(0))\notin\widetilde{\Sigma}_\gamma$, then $\lim_{t\to t_c^-}\eta(t) =\infty$. And from \eqref{auxnonlin1}, we obtain $\lim_{t\to t_c^-}\xi(t) =-\infty$. 
This completes the proof to the proposition. 
\end{proof}
For the case when $c(x)\equiv c$ in \eqref{EPMain}, we can apply Propostion \ref{phpl} with $\gamma=c$ to \eqref{nonlin} and immediately obtain Theorem \ref{cconst}.

\subsection{Comparison lemma} 
We are now in position to present the comparison result. 
\begin{lemma}[Comparison lemma]\label{comparison} Let $(\rho, d)$ be the solution to \eqref{nonlin}, and $(\eta, \xi)$
be solution of \eqref{auxnonlin}.  Then as long as these solutions exist, we have\\ 
(i) For $\gamma=\min_\mathbb{T} c(x)$:  If $\rho(0)<\eta(0)$ and $d(0)>\xi(0)$, then 
$$
\rho(t)<\eta(t),\quad d(t) >\xi(t).
$$
(ii) For $\gamma=\max _\mathbb{T} c(x)$:  If $\rho(0)> \eta(0)$ and $d(0)< \xi(0)$, then 
$$
\rho(t)>\eta(t),\quad d(t) <\xi(t).
$$
\end{lemma}
\begin{proof}
We will show $(i)$. Similar arguments follow for $(ii)$. We will argue by contradiction. To this end, let $t_1$ be the first time when the result is violated. Integrate \eqref{nonlin11} and \eqref{auxnonlin1} respectivley to get $\rho(t) = \rho(0)e^{-\int_0^t d\, d\tau}$ and $\eta(t) = \eta(0)e^{-\int_0^t\xi\, d\tau}$. Since $\int_0^{t_1}d\, d\tau > \int_0^{t_1}\xi\, d\tau $ and $\rho(0)<\eta(0)$, we obtain 
$$
\rho(t_1) = \rho(0)e^{-\int_0^{t_1} d\, d\tau}<\eta(0)e^{-\int_0^{t_1} \xi\, d\tau} = \eta(t_1).
$$
We can then conclude that $d(t_1) = \xi(t_1)$ along with $\rho(t_1)<\eta(t_1)$. Subtracting \eqref{auxnonlin2} from \eqref{nonlin12}, we obtain,
$$
(d-\xi)' = -(d+\xi+\nu)(d-\xi) + k(\rho-\eta) -k(c-\gamma).
$$
Plugging in $t=t_1$ and taking $\gamma = \min c$, we obtain
\begin{align*}
(d-\xi)'(t_1) & = k(\rho(t_1)-\eta(t_1))-k(c-\min c)\\
& >0.
\end{align*}
This is a contradiction because this implies that for all $t<t_1$ sufficiently close, $d(t)<\xi(t)$ and hence, $t_1$ cannot be the first time of violation. 
\end{proof}
\subsection{Proofs of Theorems \ref{GS} and \ref{FTB}} Using the tools developed above, we are now ready to prove our main results.\\ 
\textit{Proof of Theorem \ref{GS}:} Consider \eqref{nonlin} along a fixed characteristic and \eqref{auxnonlin} for $\gamma = c_1$. From hypothesis of theorem, we see that $d(0)>\frac{\lambda_1}{c_1}(\rho(0)-c_1)$. As a result we can choose $\eta(0)>\rho(0)$ and $\xi(0)<d(0)$ such that
$$
d(0)>\xi(0)\geq \frac{\lambda_1}{c_1}(\eta(0)-c_1)> \frac{\lambda_1}{c_1}(\rho(0)-c_1).
$$
Applying Lemma \ref{comparison}, we obtain
$$
\rho(t)<\eta(t),\ and\ d(t)>\xi(t), 
$$
for as long as these functions exist. Using Propostion \ref{phpl}, we obtain that $\rho(t)<\eta(0)e^{\lambda_1 t}$ and $d(t)>-\lambda_1$ for all $t>0$. However, note that $\eta(0)$ can be chosen to be greater than but arbitrarily close to $\rho(0)$ and all the above arguments still hold. Therefore, in the limit, we have $\rho(t)\leq \rho(0)e^{\lambda_1 t}$. Also, a uniform upper bound on $d$ can be obtained in a similar fashion as in Proposition \ref{phpl}. From \eqref{nonlin12},
\begin{align*}
d'  & \leq -d^2-\nu d-kc_2\\
& = -(d+\Omega(c_2,\nu))(d-\Theta(c_2,\nu)).
\end{align*}
Hence, $d(t)\leq \max\{ d(0),\Theta(c_2,\nu) \}$. \\
Collecting all the characterstics, we finally obtain,
\begin{align*}
& ||\rho(t,\cdot)||_\infty\leq ||\rho_0||_\infty e^{\lambda_1 t},\\
& -\lambda_1\leq u_x(t,\cdot)\leq \max\{ || u_{0x}||_\infty, \Theta(c_2,\nu) \}.
\end{align*}
This concludes the proof of Theorem \ref{GS}.\qed\\

\textit{Proof of Theorem \ref{FTB}:} Consider \eqref{nonlin} with $\rho(0) = \rho_0(x_0)$, $d(0)=u_{0x}(x_0)$ for $x_0$ as in the statement of the theorem. In \eqref{auxnonlin}, let $\gamma = c_2$. From the hypothesis of the theorem, we see that $d(0)<\frac{\lambda_2}{c_2}(\rho(0)-c_2)$. We can then choose $\eta(0)<\rho(0)$ and $\xi(0)>d(0)$ such that,
$$
d(0)<\xi(0)<\frac{\lambda_2}{c_2}(\eta(0)-c_2)<\frac{\lambda_2}{c_2}(\rho(0)-c_2).
$$
Applying Lemma \ref{comparison}, we obtain
$$
\rho(t)>\eta(t),\ and\  d(t)<\xi(t). 
$$
Using Propostion \ref{phpl}, we obtain that for some $t_c>0$, $\rho(t)\to\infty$ as $t\to t_c^-$. And from \eqref{nonlin11}, $\lim_{t\to t_c^-} d(t) = -\infty$ and the solution ceases to be $C^1$. \qed

\section{EPA systems with variable background}
\label{varbackalign}
\subsection{Reformulations} 
Set $G :=u_x+\nu +\psi\ast\rho$. Taking derivative of $G$ along 
$$
	' = \frac{\partial}{\partial t} + u\frac{\partial}{\partial x},
$$
we have,
\begin{align*}
G' & = (u_t)_x + \psi\ast\rho_t + u(u_x + \psi\ast\rho)_x\\
& = (-uu_x -\nu u + \psi\ast(\rho u)-u\psi\ast\rho -k\phi_x)_x-(\psi\ast(\rho u))_x + u(u_x + \psi\ast\rho)_x\\
& = -G u_x + k(\rho - c)\\
& = -G(G -\nu -\psi\ast\rho) + k(\rho -c). 
\end{align*}
We used \eqref{backalign} to obtain the second and third equations. Consequently, along the particle path given by,
$$
	\Gamma=\{(t,x)|\;  x'(t)=u(t,x(t)), x(0)=\alpha \in \mathbb{T}\},
$$
we get the following ODE system,
\begin{subequations}\label{nonlin2}
	\begin{align}
	& \rho' = -\rho(G-\nu-\psi\ast\rho), \label{nonlin21} \\
	& G' = -G(G-\nu-\psi\ast\rho) + k(\rho -c(x(t)) ), \label{nonlin22}
	\end{align}
	\end{subequations}
with initial condition,
$$
\rho(0)=\rho_0(\alpha ),\quad G(0)= u_{0x}(\alpha )+\nu+\psi\ast\rho_0(\alpha).
$$
The roadmap to the proofs of the main theorems will be similar to the previous section. However, due to the addition of the non local term, here  we first transform the ODE system \eqref{nonlin2} into a simple system, and then introduce an auxiliary ODE system which can be used for comparison. 
And eventually we use these tools to prove our main results.  
 
Note that the transformation will require $\rho(t)>0$ as long as the solution exists, this is ensured by assuming $\rho(0)>0$. In fact,  from \eqref{nonlin21}, we have that $\rho$ maintains sign, hence the zero case can be handled separately.
 
Next, we use the following transformation of variables for the case $\rho>0$, 
\begin{align}
\label{transformation2}
w = \frac{G}{\rho},\qquad s = \frac{1}{\rho},
\end{align}
to derive an ODE system for $w$ and $s$. Differentiating $w$,
\begin{align*}
\left(\frac{G}{\rho}\right)' & = -\frac{\rho'}{\rho^2} +\frac{G'}{\rho}\\
& = \frac{(G-\nu-\psi\ast\rho)G}{\rho} - \frac{(G-\nu-\psi\ast\rho)G}{\rho} + k\frac{(\rho - c)}{\rho} \\
& = k-kcs. 	
\end{align*}
Likewise, we differentiate $s$,
\begin{align*}
\left(\frac{1}{\rho}\right)' & = \frac{(G-\nu-\psi\ast\rho)}{\rho}\\
& = w - \nu s-s\psi\ast\rho. 	
\end{align*}
We then obtain the following ODE system,
\begin{subequations}
\label{lin2}
\begin{align}
& w' = k - kcs, \label{lin21} \\
& s' = w-(\nu +\psi\ast\rho)s, \label{lin22} 
\end{align}
\end{subequations}
with initial conditions 
$$
w(0) := \frac{G(0)}{\rho(0)} \;  \text{and} \;  s(0):=\frac{1}{\rho(0)}.
$$ 
\subsection{Threshold analysis for the auxiliary system} 
Corresponding to \eqref{lin2}, we introduce the following auxiliary system,
\begin{subequations}
\label{auxlin2}
\begin{align}
& p' = k - k\gamma q, \label{auxlin21} \\ 
& q' = p - \beta q, \label{auxlin22}
\end{align}	
\end{subequations}
where $\gamma\geq 0$, $\beta\geq\nu$ are parameters and initial conditions $(p(0),q(0)>0)$. Hence, $p ,q$ are functions of time as well as the parameters $\gamma ,\beta $. However, we will omit the latter dependence on the parameters whenever it is clear from context. We have the following proposition.
\begin{proposition}
\label{phpl2}
For the system \eqref{auxlin2}, with initial conditions $(p(0),q(0)>0)$, we have that $q(t)>0$ for all $t>0$ if and only if 
$$
p(0)\geq \frac{\lambda}{\gamma} - \mu q(0),
$$
where $\lambda = \Omega (\gamma,\beta)$ and $\mu = \Theta(\gamma,\beta)$. Additionally, if the above inequality holds, then it holds for all times, i.e.,
$$
p(t)\geq \frac{\lambda}{\gamma}-\mu q(t),\quad \forall t>0.
$$
\end{proposition}
We will once again make use of the phase plane analysis technique developed in \cite{BL20}.
\begin{proof}
\eqref{auxlin2} is a linear system with critical point $(\beta /\gamma ,1/\gamma )$. Written in matrix form, the system is:
$$
\begin{bmatrix}
p - \frac{\beta}{\gamma}\\
q-\frac{1}{\gamma}	
\end{bmatrix}' = 
\begin{bmatrix}
0 & -k\gamma \\
1 & -\beta	
\end{bmatrix}
\begin{bmatrix}
p - \frac{\beta}{\gamma}\\
q-\frac{1}{\gamma}	
\end{bmatrix}.
$$
The eigenvalues of the coefficient matrix are $-\lambda$ and $\mu$ and the general solution to the system is,
\begin{align}
\label{auxlinsys2sol}
\begin{aligned}
& \begin{bmatrix}
p - \frac{\beta}{\gamma}\\
q-\frac{1}{\gamma}	
\end{bmatrix} = A
\begin{bmatrix}
k\gamma\\
\lambda	
\end{bmatrix} e^{-\lambda t}
+ B
\begin{bmatrix}
-k\gamma\\
\mu	
\end{bmatrix} e^{\mu t}.	
\end{aligned}
\end{align}
From the flow of solution trajectories we see that the separatrix with incoming trajectories serves to divide the upper half plane ($q>0$) into two invariant regions, one of which has the property that if $q(0)>0$,  then $q(t)>0$ for all $t>0$. 

Such separatrix corresponds to the special solutions with $B=0$, i.e.,
\begin{align*}
& \begin{bmatrix}
p - \frac{\beta}{\gamma}\\
q-\frac{1}{\gamma}	
\end{bmatrix} = A
\begin{bmatrix}
k\gamma\\
\lambda	
\end{bmatrix} e^{-\lambda t}.	
\end{align*}
Consequently, this trajectory equation is, 
$$
\lambda p = \frac{\lambda\beta}{\gamma} - k + k\gamma q. 
$$ 
Note that $\frac{1}{\lambda} = -\frac{\mu}{k\gamma}$ and $\beta+\mu=\lambda$, the above equation 
becomes 
\begin{align*}
p 
 = \frac{\lambda}{\gamma}-\mu q.
\end{align*}
\begin{figure}[H]
 \centering
 \includegraphics[width=0.6\linewidth]{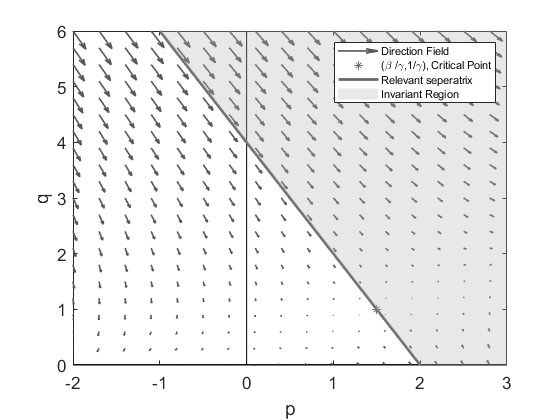}
 \caption{Direction field for linear system along with the invariant region.($\beta =1.5, k=-1,\gamma =1$)}
 \label{fig1}
\end{figure}
Thus the above mentioned region can be characterized by 
$$
\Sigma_{\gamma,\beta} := \left\{(p,q):  p\geq \frac{\lambda}{\gamma} -\mu q,\ q> 0 \right\}.
$$
Now suppose $(p(0),q(0)>0)\notin\Sigma_{\gamma,\beta}$. 
Since, the linear ODE system has only one critical point, we have that $\lim_{t\to\infty}(|p(t)|,q(t))= (\infty,-\infty)$. Hence, the solution crosses $q=0$ line at some finite time, $t_c$. This by itself concludes the proof but we will, however, derive an upper bound on $t_c$ using the general solution \eqref{auxlinsys2sol},
$$
q(t) = \frac{1}{\gamma} + \left(\frac{-p(0)+\lambda q(0) - \frac{\mu}{\gamma}}{\lambda+\mu}\right)e^{-\lambda t} + \left(\frac{p(0)+\mu q(0)-\frac{\lambda}{\gamma}}{\lambda+\mu}\right)e^{\mu t}.
$$   
Assuming $(p(0),q(0))\notin\Sigma_{\gamma,\beta}$, we have
\begin{align*}
\eta(t) & = 	\frac{1}{\gamma} + \left(\frac{-p(0)+\lambda q(0) - \frac{\mu}{\gamma}}{\lambda+\mu}\right)e^{-\lambda t} - \frac{\left|p(0)+\mu q(0)-\frac{\lambda}{\gamma}\right|}{\lambda+\mu}e^{\mu t}\\
& \leq \frac{1}{\gamma} +\frac{|p(0)|}{\lambda+\mu} +q(0) - \frac{\left|p(0)+\mu q(0)-\frac{\lambda}{\gamma}\right|}{\lambda+\mu}e^{\mu t}.
\end{align*}
Hence, $q(t_c) = 0$ for some
$$
t_c\leq \frac{1}{\mu}\ln\left( \frac{(\lambda+\mu)(q (0)+ \gamma^{-1}) +|p(0)|}{\left|p(0)+\mu q(0)-\lambda\gamma^{-1}\right|} \right).
$$ 
\end{proof}
\subsection{Comparison lemma}
We will now derive the comparison lemma.
\begin{lemma}[Comparison Lemma]
\label{comparison2}
Let $(w,s)$ be solution to \eqref{lin2} and $(p,q)$ be solution to \eqref{auxlin2}. Then as long $s\geq 0$, we have:\\ 
(i) For $c=c_1$, $\beta = \nu+\psi_M$: If $s(0)>q(0)$ and $w(0)>p(0)$, then 
$$
s(t)>q(t),\qquad w(t)>p(t).
$$
(ii) For $c=c_2$, $\beta = \nu+\psi_m$: If $s(0)<q(0)$ and $w(0)<p(0)$, then 
$$
s(t)<q(t),\qquad w(t)<p(t).
$$
\end{lemma}
\begin{proof}
We only prove the first assertion. Second assertion can be proved by similar arguments. We argue by contradiction: let $t_1$ be the first time at which statement (i) is violated. Subtracting \eqref{lin21} from \eqref{auxlin21}, and integrating we obtain,
\begin{align*}
w(t)-p(t) & = w(0)-p(0) - k\!\!\int_0^t (cs-\gamma q) d\tau\\
& = w(0)-p(0)  -k\gamma\!\!\int_0^t\!\!(s-q)\, d\tau - k\!\!\int_0^t\!\!\! s(c-\gamma)\, d\tau.
\end{align*}
Taking $\gamma = c_1=\min_{\mathbb{T}} c(x)$ and plugging in $t=t_1$ in the equation obtained, we have that 
$$
w(t_1)- p(t_1) \geq w(0)-p(0)  -k\gamma\!\!\int_0^{t_1} \!\!(s-q)\, d\tau>0. 
$$ 
Therefore, the only possibility left is that $s(t_1) = q(t_1)$.

Subtracting \eqref{lin22} from \eqref{auxlin22}, we obtain
\begin{align*}
(s-q)' & = (w-p) + \beta\eta - s(\nu + \psi\ast\rho)\\
& = (w-p) + \beta (q - s) + s(\beta - \nu-\psi\ast\rho).
\end{align*}
Note that $\psi\ast\rho\in [\min_\mathbb{R}\psi,\max_\mathbb{R} \psi] =  [\psi_m,\psi_M]$. Taking $\beta = \nu +\psi_M$ and plugging in $t=t_1$ in the above equation, we get 
$$
(s-q)'(t_1) \geq w(t_1)-p(t_1)>0. 
$$ 
This means that for $t<t_1$ sufficiently close,  we must have $s(t)<q(t)$, which is a contradiction. 
\end{proof}
\subsection{Proofs of Theorems \ref{nonlocvarcgs} and \ref{nonlocvarcftb} }

As usual, we will analyze the solution on a single characteristic and since the inequality in the statement of the theorem holds for all $x$, we can then collect all the characteristics to conclude the result. 

Therefore, it suffices to obtain the thresholds results for \eqref{nonlin2} using Proposition \ref{phpl2} and Lemma \ref{comparison2}. 

First we show that $G$ is always bounded form above irrespective of the choice of the initial data.  From \eqref{nonlin22}, we have
\begin{align}
\label{Gequality}
\begin{aligned}
G' & \leq -G(G-\nu -\psi\ast\rho) - kc\\
& = -(G^2 - (\nu +\psi\ast\rho)G + kc) \\
& = -(G-G_+)(G-G_-),
\end{aligned}
\end{align}
where 
$$
G_+= \Omega(c, \nu+\psi\ast\rho), \quad G_-=-\Theta(c, \nu+\psi\ast\rho)
$$ 
depend on $c$ and $\psi*\rho$, therefore changing in time.  
 
Note that 
$$
G_+ \leq  \Omega(c_2, \nu+\psi_M)
$$
and the fact that $G$ is non-increasing in the regime where $G \geq G_+$,  hence 
\begin{align*}
G \leq \max\{G(0),\sup G_+ \}  \leq  \max\left\{ u_x(0)+\nu+\psi_M,  \Omega(c_2, \nu+\psi_M)\right\}.
\end{align*}
Hence,
\begin{align*}
u_x &\leq \sup G  -\nu -\min\psi\ast\rho\\
& \leq \max\left\{ u_x(0), \Omega(c_2, \nu+\psi_M)-\nu-\psi_M \right\} +\psi_M- \psi_m\\
& = \max\left\{ u_x(0),  \Theta(c_2,\nu+\psi_M) \right\} +\psi_M- \psi_m
\end{align*}
Note that this upper bound holds irrespective of the hypothesis of the theorem. $u_x$ being upper bounded is a result of the dynamics of the system \eqref{nonlin2}.

We now handle the $\rho(0)=0\equiv \rho$ case before dealing with the case $\rho>0$ separately. 

In such case we have $\rho\equiv 0$. Consider \eqref{nonlin2} with $\rho(0) = \rho_0(\alpha)$,
 $G(0) = u_{0x}(\alpha)+\nu+\psi\ast\rho_0(\alpha)$ with a fixed $\alpha \in \mathbb{T}$.  
Hence along the characteristics starting from $\alpha$ we have 
\begin{align}
\label{Gequality+}
\begin{aligned}
G' & = -G(G-\nu -\psi\ast\rho) - kc\\
& = -(G-G_+)(G-G_-),
\end{aligned}
\end{align}
where $G_{\pm}$ are same as  above.   From phase line analysis, we have that 
$$
G(t) \geq \sup G_-
$$ 
for all $t>0$ if 
$$
G(0)\geq\sup G_- = - \Theta(c_1, \nu +\psi_M). 
$$
We will show that this indeed satisfies the threshold inequality in the theorem.
\begin{align*}
u_x(0) & = G(0)-\nu-\psi\ast\rho(0)\\
& > -\Theta(c_1,\nu+\psi_M)-\nu-\psi\ast\rho(0)\\
& = \sup G_- -\nu - \psi\ast\rho(0),
\end{align*}
then 
$$
u_x(t) = G(t)-\nu-\psi\ast\rho(t) \geq -\Theta(c_1,\nu+\psi_M)-\nu-\psi_M= -\lambda_M
$$ 
for all $t> 0$.

On the other hand, consider \eqref{nonlin2} with $\alpha=x_0$ as in the statement of the Theorem \ref{nonlocvarcftb}. Then from \eqref{nonlin22}, 
$$
G' = -(G-G_+)(G-G_-).
$$
From phase line analysis, we have that $G \to -\infty$ in finite time if 
$$
G(0)<\inf G_- = -\Theta(c_2, \nu+\psi_m).
$$
Hence, if
\begin{align*}
u_x(0) & = G(0)-\nu-\psi\ast\rho(0)\\
& < -\Theta(c_2,\nu+\psi_m) - \nu -\psi\ast\rho(0)
\end{align*}
then $\lim_{t\to t_c^-} G=\lim_{t\to t_c^-} u_x = -\infty $ for some time $t_c$ and this is indeed the statement of Theorem \ref{nonlocvarcftb}.

Now we deal with the case when $\rho>0$.

\noindent \textit{Proof of Theorem \ref{nonlocvarcgs}:}  Along the a fixed characteristics from $\alpha$, we rewrite 
the initial threshold condition in the theorem as,
$$
\frac{G(0)}{\rho(0)}>\frac{\lambda_M}{c_1} - \frac{\mu_M}{\rho(0)}, \quad \mu_M:=\Theta(c_1, \nu+\psi_M), 
$$
and this when transformed by \eqref{transformation2}, reads
$$
w(0)>\frac{\lambda_M}{c_1}-\mu_M s(0).
$$
We can then choose $p(0)<w(0)$ and $q(0)<s(0)$ in \eqref{auxlin2} such that the following holds,
$$
w(0)>p(0)\geq \frac{\lambda_M}{c_1} - \mu_M q(0)> \frac{\lambda_M}{c_1} - \mu_M s(0).
$$
Applying Lemma \ref{comparison2} and Proposition \ref{phpl2} for $\gamma=c_1$, $\beta = \nu+\psi_M$, we have that 
$$
w(t)>p(t)\geq \frac{\lambda_M}{c_1} - \mu_M q(t) > \frac{\lambda_M}{c_1} - \mu_M s(t).
$$
for all $t>0$ along with the positivity of $s$, i.e. $s(t)>0$.  Hence,
\begin{align*}
& w(t)> \frac{\lambda_M}{c_1} - \mu_M s(t),\quad \forall t>0.
\end{align*}
Transforming back to $(\rho, G)$, we have
$$
G(t)>\frac{\lambda_M\rho(t)}{c_1}-\mu_M.
$$
From this we can obtain a lower bound on $u_x$.
\begin{align*}
u_x & = G -\nu-  \psi\ast\rho > -\mu_M-\nu -\psi_M = -\lambda_M. 	
\end{align*}
Integrating \eqref{nonlin21}, 
\begin{align*}
\rho(t) & = \rho(0)e^{-\int_0^t u_xd\tau} < \rho(0)e^{\lambda_M t}.	
\end{align*}
Collecting all the characteristics finishes the proof of the theorem.\qed\\

\noindent\textit{Proof of Theorem \ref{nonlocvarcftb}:} 
Under the transformation \eqref{transformation2}, the initial threshold condition from the theorem reads,
$$
w(0)<\frac{\lambda_m}{c_2} - \mu_m s(0),
$$ 
where $\lambda_m:=\Omega(c_2, \nu+\psi_m)$ and $\mu_m:=\Theta(c_2, \nu+\psi_m)$. Consequently, in \eqref{auxlin2}, we can choose $p(0)>w(0)$ and $q(0)>s(0)$ such that the following holds,
$$
w(0)<p(0)<\frac{\lambda_m}{c_2}-\mu_m q(0) < \frac{\lambda_m}{c_2}-\mu_m s(0). 
$$
From Lemma \ref{comparison2}, we have that 
$$
w(t)<p(t), \quad s(t)<q(t)
$$ 
as long as $s\geq 0$.  Applying Proposition \ref{phpl2} with $\gamma = c_2$ and $\beta = \nu +\psi_m$, we have the existence of a finite time $t^*$ such that,
$$
q(t^*)=0.
$$
Therefore $s(t)$ must touch zero before $t^*$, say at $t_c<t^*$.   Consequently, $\lim_{t\to t_c^-} \rho(t) = \infty$ and therefore, from \eqref{nonlin21},
$$
 \lim_{t\to t_c^-}u_x(t,x(t,x_0)) = -\infty.
 $$
This concludes the proof.\qed

\section*{Acknowledgement} This work was partially supported by the National Science Foundation under Grant DMS181266. 

\bibliographystyle{abbrv}

\end{document}